\documentclass{article}

\usepackage{amsmath, amsthm, amsfonts, amssymb}
\usepackage{mathptmx}
\usepackage{graphicx}
\usepackage{fancybox}
\usepackage{color}
\usepackage{ulem}

\numberwithin{equation}{section}
\def\supp{\operatorname{supp}}

\def\div{\operatorname{div}}
\def\a{\alpha}
\def\e{\varepsilon}

\def\R{\bf R}

\newtheorem{thm}{Theorem}[section]
\newtheorem{df}{Definition}[section]

\newtheorem{lem}{Lemma}[section]
\newtheorem{cor}{Corollary}[section]
\newtheorem{prop}{Proposition}[section]
\newtheorem{rem}{Remark}[section]
\numberwithin{equation}{section}
\newtheorem{exam}{Example}[section]

\title{Generalized weighted Hardy's inequalities \\with compact perturbations}
\author{HIROSHI ANDO  and  TOSHIO HORIUCHI}
\date{}

\begin{document}

\maketitle

\today

\begin{abstract}
Let  $\Omega $ be  a  bounded domain of ${\R}^N \ (N\ge 1)$ with boundary of class $C^2$.
In the present paper  we  shall study  a variational problem relating 
the weighted Hardy inequalities with sharp missing terms established in \cite{H}.
As weights  we treat non-doubling 
functions of the  distance  $\delta(x)=\text{dist}(x,\partial \Omega)$ to the   boundary $\partial\Omega.$
  \footnote{ 
This research was partially supported by Grant-in-Aid for Scientific Research (No. 20K03670, No. 21K03304). }
\end{abstract}

\section{ Introduction}
Let $W({\R}_+)$ be a class of  functions
$$  \{ w(t)\in C^1({\R}_+): w(t)>0, \lim_{t\to+0}w(t)=a \ \text{for some} \ a\in [0,\infty] \}$$  with ${\R}_+=(0,\infty).$ 
For $1<p<\infty$, 
as weights of Hardy's inequalities we adopt functions $W_p(t)=w(t)^{p-1}$ with $w(t)\in P({\R}_+) \cup Q({\R}_+)$, where   
 \begin{equation}  \begin{cases}P({\R}_+)= \{ w(t)\in W({\R}_+) : \,  w(t)^{-1} \notin L^1((0,\eta)) \ \text{for some} \ \eta >0\},\\
 Q({\R}_+) =\{ w(t)\in W({\R}_+) :  \, w(t)^{-1}\in L^1((0,\eta)) \ \text{for any} \ \eta >0 \}.
 \end{cases}
 \end{equation}
Clearly $W({\R}_+)=P({\R}_+) \cup Q({\R}_+)$  and $P({\R}_+)\cap Q({\R}_+)= \emptyset$. 
(For the precise definitions see the  section 2. 
See also \cite{H}, \cite{AHL}.)  
A positive { continuous} function $w(t)$ on $\R_+$ is  said to be a  doubling weight  if  there exists a positive number $C$  such  that  we  have 
\begin{equation} C^{-1} w(t)\le w(2t)\le C w(t) \qquad\text{for all} \ \ \ t\in\R_+.\label{doubling} \end{equation}
When $w(t)$ does not possess this property, $w(t)$ is  said to be  a non-doubling weight in the  present paper.
In one-dimensional case we   typically treat a weight function $w(t)$ that  may vanish or  blow up in infinite order such  as $e^{-1/t}$ or $e^{1/t}$ at $t=0$. 
In  such cases the limit of ratio $w(t)/w(2t)$ as  $t\to +0$ may become $0$ or $+\infty$, and hence they are regarded as non-doubling weights according to our notion. \par
In \cite{H}, 
 we  have established 
 $N$-dimensional  Hardy inequalities with non-doubling
  weights  being  functions  of the  distance  
$\delta (x)=\text{dist}(x,\partial\Omega)$ to the boundary $\partial\Omega$,
  where $\Omega$   is a 
{
bounded domain of class $C^2$ in $\mathbf R^N$. 
}
In  this  paper we  shall study  a variational problem relating to those new inequalities. 

We  prepare  more  notations to  describe  our  results. 
Let $1<p<\infty$. 
For  $ W_p (t)= w(t)^{p-1}$ with $ w(t)\in W({\R}_+)$, we  define  a weight function $W_p(\delta(x))$ on $\Omega$  by $$W_p(\delta(x))=(W_p\circ\delta)(x).$$
By
  $ L^p(\Omega;W_p(\delta))$ we denote the space of Lebesgue measurable functions with weight $ W_p(\delta(x))$, 
for which 
\begin{equation} \| u \|_{ L^p(\Omega;W_p(\delta))} = \bigg( \int_{\Omega}|u(x)|^p W_p(\delta(x))\,dx\bigg ) ^{\!1/p} < +\infty.\label{2.1} \end{equation} 
 $W_{0}^{1,p}(\Omega;W_p(\delta))$ is given by the  completion of $C_c^\infty(\Omega)$ with respect to the norm 
defined by
\begin{equation} 
\| u\|_{ W_{0}^{1,p}(\Omega;W_p(\delta))} =
\| |\nabla u|  \|_{ L^p(\Omega;W_p(\delta))} + \| u\|_{ L^p(\Omega;W_p(\delta))}. \label{2.2}
\end{equation}
Then, 
 $W_{0}^{1,p}(\Omega;W_p(\delta)) $  becomes a Banach space  
with the norm $\| \cdot\|_{ W_{0}^{1,p}(\Omega;W_p(\delta))}$. 
Under these preparation we  recall the weighted Hardy inequalities in \cite{H}. 
(See Theorem \ref{CT3} and its corollary in Section 2.) 
In particular for $w(t)\in Q({\R}_+)$,
 we   have a simple inequality as Corollary  \ref{NCC1}, which  is a generalization of classical Hardy's inequality:
\begin{equation}\label{HI}
\int_\Omega  |\nabla u(x)|^p W_p(\delta(x))\,dx 
\ge \gamma \int_\Omega \frac{|u(x)|^p W_p(\delta(x))}{F_{\eta_0}(\delta(x))^p} \,dx
\end{equation}
for $u(x)\in W_{0}^{1,p}(\Omega;W_p(\delta)),$  
where $ \eta_0$ is a sufficiently small positive number, $\gamma$ is some  positive  constant and $F_{\eta_0}(t)$ is a positive function defined in Definition \ref{D2}. In particular if $w(t)=1$, then $F_{\eta_0}(t)=t \ (0<t\le \eta_0)$ and 
  (\ref{HI})  becomes  a well-known Hardy's inequality, which  is valid for  a bounded domain $\Omega $  of  ${\R}^N$ with Lipschitz boundary 
(cf. \cite{BM}, \cite{D}, \cite{MMP}, \cite{MS}).   
Further
if $\Omega$ is  convex,  then   $\gamma = \Lambda_{p}:= (1-1/p)^p$ holds  for   arbitrary $1<p<\infty$ (see \cite{MS}). \par 
\par\medskip
In the  present paper we consider the following variational problem relating 
 the general Hardy's inequalities 
   established in \cite{H}. 
For  $\lambda \in \R$, $W_p(t)= w(t)^{p-1}$ and  $ w(t)\in W_A({\R}_+)\,(\subset W(\R_+))$,
the  following variational problem (\ref{J}) can be   associated with (\ref{HI}): 
\begin{equation} 
J^w_{p,\lambda} = \inf_{ u\in W_{0}^{1,p}(\Omega;W_p(\delta))\setminus\{0\} } 
\chi^w_{p,\lambda} (u), \label{J}
\end{equation}
where
\begin{equation}  \chi^w_{p,\lambda} (u)=\frac{ \int_\Omega |\nabla u(x)|^p W_p(\delta(x))\,dx -\lambda \int_\Omega|u(x)|^p W_p(\delta(x))\,dx }{\int_\Omega |u(x)|^p W_p(\delta(x)) 
/F_{\eta_0}(\delta(x))^pdx }.\label{chi}
\end{equation}
Here $W_A({\R}_+)=P_A ({\R}_+)\cup Q_A({\R}_+)$ is a subclass of $W(\R_+)$ defined by  Definition \ref{admissibility}
and $\eta_0$ is a sufficiently small positive  number such  that the Hardy inequalities in Theorem \ref{CT3}  and Corollary \ref{NCC1} are  valid. Note  that $ J^w_{p,{0}}$ gives  the  best constant in (\ref{HI}),   the  function $\lambda\mapsto  J^w_{p,\lambda}$ is  non-increasing on $\R$ and $ J^w_{p,\lambda} \to -\infty$  as $\lambda\to \infty$. 
\par
When $p=2$ and $w(t)=1$,   this variational problem   (\ref{J})  was originally studied in \cite{BM}. Then, the problem
  (\ref{J}) was intensively  studied in  \cite{ah3} in the case  that  $1<p<\infty$ 
and $w(t)=t^{\alpha p/(p-1)}\in Q_A(\R_+)$ with $\alpha<1-1/p$. 
In this paper we  further investigate {the variational problem (\ref{J}) with  non-doubling weight }functions
 $w(t)\in W_A(\R_+)$ and  we make clear the attainability of the infimum $ J^w_{p,\lambda}$ as Theorem \ref{Main} and Theorem \ref{p=2}.
\par\medskip

 This  paper is  organized in  the  following way: In Subsection 2.1 we introduce a  class of {weight} functions  $W(\mathbf R_+)$ and two subclasses $P(\mathbf R_+)$  and $ Q(\mathbf R_+)$ together with 
so-called Hardy functions, which are crucial  in this  paper. Further a notion of  admissibilities for $P(\mathbf R_+)$  and $ Q(\mathbf R_+)$ is  introduced. In  Subsection 2.2, we recall the weighted Hardy's inequalities in \cite{H} which are crucial in this work.
In Section 3, the main  results are described.  
Theorem \ref{Main} and Theorem \ref{p=2} are established  in   Section 4 and  Section 5  respectively.

\section{Preliminaries }
\subsection{Weight functions }

First we  introduce a class of weight functions according to \cite{H} which  is crucial in this paper.
\begin{df}\label{D0}
 Let us  set ${\R}_+= (0,\infty)$ and 
  \begin{equation}
W({\R}_+)=\{ w(t)\in C^1({\R}_+): w(t)>0, \lim_{t\to+0}w(t)=a \ \text{for some} \ a\in [0,\infty] \}.\end{equation}
\end{df}
In  the  next  we define two subclasses of  $W({\R}_+)$.

\begin{df}\label{D1} Let  us  set 
 \begin{align} 
P({\R}_+)&= \{ w(t)\in W({\R}_+) : \,  w(t)^{-1} \notin L^1((0,\eta)) \ \text{for some} \ \eta >0\},
\\[0.5ex]
Q({\R}_+)&=\{ w(t)\in W({\R}_+) :  \, w(t)^{-1}\in L^1((0,\eta)) \ \text{for any} \ \eta >0 \}.\end{align}
\end{df}
Here we give   fundamental examples:
\begin{exam} 
\begin{enumerate}
\item 
$t^{\a } \in P(\mathbf R_+)$ if $\a \ge 1$  and 
$t^{\a } \in Q(\mathbf R_+)$ if $\a < 1$.
\item $ e^{-1/t}\in P(\mathbf R_+)$ and $ e^{1/t}\in Q(\mathbf R_+)$. 
\item For $ \a\in \mathbf R$, $ t^\a e^{-1/t}\in P(\mathbf R_+)$ and $t^\a e^{1/t}\in Q(\mathbf R_+)$.
\end{enumerate}
\end{exam}

\begin{rem} \begin{enumerate}
\item  $W({\R}_+)= P({\R}_+)\cup Q({\R}_+)$ and $P({\R}_+)\cap Q({\R}_+)= \emptyset$ hold.
\item If  $w(t)^{-1} \notin L^1((0,\eta))$ for some $\eta >0$, 
then  $w(t)^{-1} \notin L^1((0,\eta))$ for any $\eta >0$. 
Similarly if $w(t)^{-1} \in L^1((0,\eta))$ for some $\eta >0$, 
then  $w(t)^{-1} \in L^1((0,\eta))$ for any $\eta >0$.
\item
If  $w(t)\in  P({\R}_+)$, then $\lim_{t\to +0}w(t)=0$. Hence 
by  setting $w(0)=0$, $w(t)$ is uniquely  extended to  a continuous function on $[0,\infty)$. 
On the other hand  if $w(t)\in  Q({\R}_+)$, then    possibly  $\lim_{t\to +0}w(t)=+\infty$.
\end{enumerate}
\end{rem}

In the next we define  {functions such as $F_{\eta}(t)$ and $G_{\eta}(t)$ in order to introduce  variants of the Hardy potential like
 $F_{\eta_0}(\delta(x))^{-p}$ in (\ref{HI}).}

\begin{df}\label{D2} Let $\mu>0$ and   $\eta>0$. 
For $w(t) \in W({\R}_+)$, 
we  define the followings:
 \begin{enumerate}
 \item  When $w(t)\in P({\R}_+)$, 
   \begin{align}
 F_{\eta}(t; w,\mu) & = \begin{cases} 
w(t)\left( \mu+ \int_t^\eta {w(s)}^{-1}\, {ds}\right) &  \text{ if } \ t\in (0,\eta), \\ 
w(\eta) \mu &  \text{ if } \ t\ge  \eta, 
 \end{cases}\\
 G_\eta(t;w,\mu) &= \begin{cases} 
\mu+  \int_t^\eta  {F_\eta(s;w,\mu)}^{-1}\, {ds} &  \text{ if } \ t\in (0,\eta),\\
\mu & \text{ if } \ t\ge \eta.
 \end{cases}\label{PG}
\end{align}

\item When $w(t)\in Q({\R}_+)$, 
 \begin{align}
 F_{\eta}(t; w) &= 
\begin{cases} w(t)\int_0^t {w(s)}^{-1} \,{ds} & \text{ if } \ t\in (0,\eta), \\ 
 w(\eta) \int_0^\eta w(s)^{-1} \,ds & \text{ if } \ t\ge \eta, 
\end{cases}\\
  G_\eta(t;w,\mu)&= 
\begin{cases} \mu+  \int_t^\eta  {F_\eta(s;w)}^{-1}\, {ds} & \text{ if } \ t\in (0,\eta), \\
\mu & \text{ if } \ t\ge \eta.
 \end{cases}
\end{align}
\item $F_{\eta}(t;w,\mu)$ and $F_{\eta}(t;w)$ are abbreviated as $F_{\eta}(t)$.
$ G_{\eta}(t; w,\mu)$ is abbreviated as  $ G_{\eta}(t)$.
 \item  For   $ w(t)\in P({\R}_+)$  or $Q({\R}_+)$, we  define
 \begin{equation}  W_p(t)= w(t)^{p-1}. \label{1.4}\end{equation}
\end{enumerate}
\end{df}

\begin{rem}\label{defofG} In the definition (\ref{PG}), one can replace $G_\eta(t; w,\mu)$ with the more  general 
$G_\eta(t; w,\mu,\mu')= \mu'+  \int_t^\eta  {F_\eta(s;w,\mu)}^{-1}ds$ if $t\in (0,\eta)$,  
$G_\eta(t; w,\mu,\mu')=\mu'$ if $t\ge \eta$ with $\mu'>0$. 
However, for simplicity  this  paper uses  (\ref{PG}).  
\end{rem}


Here we give fundamental examples:

\begin{exam} Let $w(t)=t^{\a } $ for $\a\in \R$. 
\begin{enumerate}
\item When $\a>1$, $F_\eta(t)=t/(\a -1) \,$   and 
$G_\eta(t)=\mu +(\a -1)\log(\eta/t)$ for $t\in(0,\eta)$ 
provided that  $ \mu= \eta^{1-\a }/(\a -1)$.
\item When $\a=1$,  $F_\eta(t)=t(\mu+ \log(\eta/t)) \,$   and 
$G_\eta(t)=\mu -\log \mu +\log\left( \mu+ \log(\eta/t)\right)$ for $t\in(0,\eta)$.
\item When $\a<1$, $F_\eta(t)=t/(1-\a ) \,$   and 
$G_\eta(t)=\mu +(1-\a )\log(\eta/t)$ for $t\in(0,\eta)$.
\end{enumerate}
\end{exam}
By using integration by parts  we  see  the followings: 
\begin{exam}
\begin{enumerate}
\item When either $w(t)=e^{-1/t} \in P(\R_+)$ or $w(t)=e^{1/t}\in Q(\R_+)$, 
we have $F_\eta(t)= O(t^2)$ as $t\to +0$.
\item Moreover,
if $w(t)= \exp({\pm t^{-\a}}) $ with $\a> 0$, then $F_\eta(t)= O(t^{\a+1})$ as $t\to +0$. 
{ In fact, it  holds  that $\lim_{t\to +0} F_\eta(t)/t^{\a+1}=1/\a$. }
\end{enumerate}
\end{exam}

In a similar way  we define the following:
\begin{df}\label{fg} Let $\mu>0$ and   $\eta>0$. 
For $w(t) \in W({\R}_+)$, 
we  define the followings:
 \begin{enumerate}
 \item When $w(t)\in P({\R}_+)$,
  \begin{align}
f_{\eta}(t; w,\mu) &= 
\begin{cases} \mu+ \int_t^\eta {w(s)}^{-1}\, {ds} & \text{ if } \ t\in (0,\eta),\\ 
 \mu &  \text{ if } \ t\ge \eta.
\end{cases}
\end{align}
 \item When $w(t)\in Q({\R}_+)$,
  \begin{align}
f_{\eta}(t; w) &= 
\begin{cases}
\int_0^t {w(s)}^{-1} \,{ds} &    \text{ if } \ t\in (0,\eta),\\
\int_0^\eta w(s)^{-1}\,ds &  \text{ if } \ t\ge \eta.
\end{cases}
\end{align}
\item $f_{\eta}(t;w,\mu)$ and $f_{\eta}(t;w)$ are abbreviated as $f_{\eta}(t)$.
\end{enumerate}
\end{df}
\begin{rem}\label{remark2.2}
\begin{enumerate}\item
We note that for  $t\in(0,\eta)$
\begin{equation}\label{2.13}\begin{cases}
 \frac {d}{dt}\log f_{\eta}(t)= -{F_\eta(t) }^{-1} & \text{ if } \ w(t)\in P({\R}_+),\\ 
\frac {d}{dt}\log f_{\eta}(t)={F_\eta(t) }^{-1} & \  \mbox{ if } \ w(t)\in Q({\R}_+),\\
\frac {d}{dt}\log G_{\eta}(t)= -({F_\eta(t) G_{\eta}(t)})^{-1},
\\
\frac {d}{dt}G_{\eta}(t)^{-1}= ({F_\eta(t) G_{\eta}(t)^2})^{-1} 
& \ \mbox{ if } \ w(t)\in W({\mathbf R}_+).
\end{cases}
\end{equation}
By Definition \ref{D2}, Definition \ref{fg}  
and (\ref{2.13}), we  see that $F_\eta(t)^{-1}\notin L^1((0,\eta))$, 
$\lim_{t\to+0}G_\eta(t)=\infty$ and $( F_\eta (t)G_\eta(t))^{-1}\notin L^1((0,\eta))$,
but $ (F_\eta(t) G_\eta(t)^2 )^{-1}\in L^1((0,\eta))$.   
\item
If $w(t)\in W(\R_+)$, 
then  we  have $\liminf_{t\to +0}F_\eta(t)=\liminf_{t\to+0}F_\eta(t)G_\eta(t)=0$
from 1. 
\end{enumerate}
\end{rem}
\begin{exam}
If either $w(t)=t^2e^{-1/t}\in P(\mathbf R_+)$ or $w(t)=t^2e^{1/t}\in Q(\mathbf R_+)$, 
then $F_{\eta}(t)=O(t^2)$ and  $G_{\eta}(t)=O(1/t)$ as $t\to +0$.
\end{exam}
Now  we introduce  two admissibilities for $P({\R}_+)$ and  $Q({\R}_+)$.
\begin{df} \begin{enumerate}\item
A function $w(t)\in P({\R}_+)$ is said  to be admissible if there exist  positive  numbers $\eta$  and $K$ such  that we  have 
 \begin{equation}\label{cond1}
 \int_t^\eta w(s)^{-1} \,ds \le e^{K/\sqrt t} \qquad\text{for} \ \ \ t\in(0,\eta).
 \end{equation}
 \item 
 A function $w(t)\in Q({\R}_+)$ is said  to be admissible if there exist  positive  numbers $\eta$  and $K$ such  that we  have 
\begin{equation}\label{cond2}
 \int_0^t w(s)^{-1} \,ds \ge e^{-K/\sqrt t} \qquad\text{for} \ \ \ t\in(0,\eta).
 \end{equation}
 \end{enumerate}
  \end{df}
\begin{df}  \label{admissibility}
By $P_A ({\R}_+)$  and $ Q_A({\R}_+)$ we  denote  the  set of  all  admissible  functions  in $P({\R}_+)$ and $ Q({\R}_+)$ respectively. We  set \begin{equation}W_A({\R}_+)= P_A ({\R}_+)\cup Q_A({\R}_+).\end{equation}
\end{df}
\begin{rem}
If   $w(t)\in W_A (\mathbf R_+)$, then 
there exist positive numbers $\eta$ and $K$ such  that we  have 
\begin{equation}
\sqrt t \,G_\eta (t) \le K\qquad\text{for} \ \ \ t\in (0,\eta).\label{2.15}
\end{equation}
{For the detail, see Proposition 2.1 in \cite{H}.}
\end{rem}
Here we give  typical examples:
\begin{exam} $e^{-1/t}\notin P_A(\mathbf  R_+)$, $e^{1/t}\notin Q_A(\mathbf  R_+)$, but $e^{-1/\sqrt{t}}\in P_A(\mathbf  R_+)$, $e^{1/\sqrt{t}}\in Q_A(\mathbf  R_+)$.\par\medskip
\par\noindent
{\bf Verifications: } \par
\noindent  $e^{-1/t} \notin  P_A(\R_+):$ For  small $t>0$, we  have $\int_t^\eta e^{1/s} \,ds \ge \int_t^{2t} e^{1/s} \,ds\ge te^{1/(2t)}$.
But this contradicts to (\ref{cond1}) for  any $K>0$. \par\noindent
{ $e^{-1/\sqrt t} \in  P_A(\R_+) :$}
Since $e^{1/\sqrt s}\le e^{1/\sqrt t}\, (t<s<\eta)$, we  have $\int_t^\eta e^{1/\sqrt s}\,ds \le \eta  e^{1/{\sqrt t}}\le e^{K/\sqrt t}$ for  some $K>1$. 
\par\noindent $e^{-1/t} \notin  Q_A(\R_+):$
For $0<s\le t$, we  have $\int_0^t e^{-1/s}\,ds\le t e^{-1/t}$. But this contradicts to (\ref{cond2}) for  any $K>0$.
 \par\noindent
{$e^{-1/\sqrt t} \in Q_A(\R_+):$} For $t/2<s<t$, we  have 
$\int_0^t e^{-1/\sqrt s}\,ds$ $\ge \int_{t/2}^t e^{-1/\sqrt s}\,ds$ 
$\ge (t/2) e^{-\sqrt{ 2/t}}$ $\ge e^{-K/\sqrt t}$ for some $ K>\sqrt 2$.
\end{exam}

\subsection{ Weighted Hardy's inequalities } 
We define a switching function.
\begin{df} {\rm (Switching function)}  
For $w(t)\in W({\R}_+)= P({\R}_+) \cup Q({\R}_+)$ we  set
\begin{equation} s(w)= \begin{cases} -1 \quad & \text{if } \quad w(t)\in P({\R}_+),\\
\,\, 1\quad & \text{if } \quad w(t)\in Q({\R}_+).\end{cases}
\end{equation}
\end{df}

Let $\Omega$ be 
{
a bounded domain of class $C^2$ in ${\R}^N$. 
}
Let $\delta(x)=\text{dist}(x,\partial\Omega)$. 
For each  small $\eta >0$, $\Omega _\eta $  and $\Sigma_\eta$ denote a tubular neighborhood of $\partial \Omega$ and $\partial(\Omega\setminus \Omega_\eta)$ respectively, namely 
\begin{equation}
\Omega_\eta = \{ x\in \Omega:  \delta(x)<\eta \}\quad\mbox{and} \quad  \Sigma_\eta = \{ x\in \Omega:  \delta(x)=\eta \} .\label{NBD}
\end{equation}

In \cite{H} we established a series of weighted Hardy's inequalities with sharp remainders. 
In particular,  
{
we have the following inequality from Theorem 3.3 in \cite{H} by noting that 
$F_\eta(t)\le F_{\eta_0}(t)$ for $\eta\in(0,\eta_0]$ and $t\in(0,\eta)$.
}

\begin{thm}\label{CT3} 
Assume  that  $\Omega$  is  a  bounded domain of  class $C^2$ in ${\R}^N$.
Assume that $1<p<\infty$ and   $ w(t)\in W_A({\R}_+)$. 
Assume  that $\mu>0$  and  $\eta_0 $  is  a sufficiently  small positive  number. 
Then, 
{
for  $\eta\in (0,\eta_0]$
} 
there exist  positive  numbers 
{
$C=C(w,p,\eta,\mu)$ and $L'=L'(w,p,\eta,\mu)$ 
}
such  that 
{
for $u(x)\in W_0^{1,p}(\Omega;W_p(\delta))$ we have 
}
\begin{align}\label{2.11}
&\int_{\Omega_\eta} 
\left( |\nabla u(x)|^p  -\Lambda_{p} \frac{|u(x)|^p}{F_{\eta_0}(\delta(x))^p}  \right)W_p(\delta(x))\,dx \nonumber \\ 
&\quad
\ge C \int_{\Omega_\eta}  \frac{|u(x)|^p {W_p(\delta(x))}}{F_{\eta}(\delta(x))^p G_{\eta}(\delta(x))^2}\,dx + s(w)L' \int_{\Sigma_\eta} |u(x)|^p W_p(\eta)\,d\sigma_\eta,
\end{align}
where $d\sigma_\eta$  denotes  surface elements on  $\Sigma_\eta$.
\end{thm}
{
Similarly we have the following inequality from Corollary 3.3 in \cite{H}.
}
\begin{cor}\label{NCC1}
Assume  that  $\Omega$  is  a  bounded domain of  class $C^2$ in ${\R}^N$.
Assume that $1<p<\infty$ and   $ w(t)\in W_A({\R}_+)$. Assume  that $\mu>0$ and  $\eta_0 $  is  a sufficiently  small positive  number. 
Then, 
{
for $\eta\in (0,\eta_0]$
}
there exist positive numbers 
{
$\gamma=\gamma(w,p,\eta,\mu)$ and $ L'=L'(w,p,\eta,\mu)$ 
}
such  that 
{
for $u(x)\in W_{0}^{1,p}(\Omega;W_p(\delta))$ we have
}
\begin{equation}\label{2.7}
\int_{\Omega} \left( |\nabla u(x)|^p  -\gamma\frac{|u(x)|^p}{F_{\eta}(\delta(x))^p} \right)W_p(\delta(x))\,dx
\ge  s(w) L' \int_{\Sigma_\eta} |u(x)|^pW_p(\eta)\,d\sigma_\eta,
\end{equation}
where $d\sigma_\eta$  denotes  surface elements on  $\Sigma_\eta$.
\end{cor}

{
\begin{rem}
In Theorem 3.3 and Corollary 3.3 in \cite{H}, it was assumed that 
$u(x)\in W_{0}^{1,p}(\Omega;W_p(\delta))\cap C(\Omega)$. 
However, since we have the inequalities (\ref{2.11}) and (\ref{2.7}) 
for $u(x)\in C_c^\infty(\Omega)$, 
by Lemma \ref{lemtrace} and Remark \ref{rem4.1} as stated later, 
we see that the inequalities (\ref{2.11}) and (\ref{2.7}) hold 
for $u(x)\in W_{0}^{1,p}(\Omega;W_p(\delta))$. 
Therefore we have Theorem \ref{CT3} and Corollary \ref{NCC1}.
\end{rem}
}

\begin{rem} 
These inequalities  are closely related to the  weighted Hardy-Sobolev inequalities with sharp remainder terms (cf. \cite{AH}, \cite{AHN}, \cite{BM}, \cite{CS}, \cite{DHA2}, \cite{AHL}, \cite{M}). 
\end{rem}

\section{Main results}

Let $\eta_0$ be a sufficiently small positive  number 
such  that the Hardy's inequalities in Theorem \ref{CT3}  and Corollary \ref{NCC1} are  valid. 
Let $w(t)\in W({\R}_+)$ and $W_p(t)= w(t)^{p-1}$ with $1<p<\infty$. 
Moreover, we assume that 
\begin{equation}\label{mono}
w'(t)\ge 0 \quad\text{for all} \ \ t\in(0,\eta_0)\quad\text{or}\quad 
w'(t)\le 0 \quad\text{for all} \ \ t\in(0,\eta_0). 
\end{equation}
Then we have the following. 
\begin{lem}\label{lem lim F=0}
Assume that $w(t)\in W(\R_+)$ satisfies (\ref{mono}). Then it holds that 
\begin{equation}\label{lim F=0}
\lim_{t\to+0}F_{\eta_0}(t)=0.
\end{equation}
In particular, $F_{\eta_0}(t)$ is bounded in $\R_+$.
\end{lem}
The proof of Lemma \ref{lem lim F=0} is stated at the end of this section. 

\medskip
For  $\lambda \in \R$, let  us  recall
the  variational problem   associated with (\ref{HI}): 
\begin{equation} 
J^w_{p,\lambda} = \inf_{ u\in W_{0}^{1,p}(\Omega;W_p(\delta))\setminus\{0\}} 
 \chi^w_{p,\lambda} (u), \label{J2}
\end{equation}
where
\begin{equation*}  \chi^w_{p,\lambda} (u)=\frac{ \int_\Omega |\nabla u(x)|^p W_p(\delta(x)) \,dx-\lambda \int_\Omega|u(x)|^p W_p(\delta(x)) \,dx}{\int_\Omega |u(x)|^p W_p(\delta(x)) /F_{\eta_0}(\delta(x))^pdx }.\label{chi}
\end{equation*}
Our main result is the following:
\begin{thm}\label{Main} 
Assume  that  $\Omega$  is  a  bounded domain of  class $C^2$ in ${\R}^N$. 
Assume  that $1<p<\infty$  and  $w(t)\in W_A({\R}_+)$ satisfies (\ref{mono}). 
Then,  there exists a constant $\lambda^\ast \in \R$ such that:
\begin{enumerate}
\item If  $\lambda \le  \lambda^\ast$,  then $J^w_{p,\lambda} = \Lambda_{p}$. 
If  $\lambda >  \lambda^\ast$,  then $J^w_{p,\lambda} < \Lambda_{p}$.
\end{enumerate}
Here 
\begin{equation}\label{Lambda}
\Lambda _{p} = \left( 1-\frac 1p\right)^{\!p}. 
\end{equation}
Moreover, it holds that:
\begin{enumerate}
\item[2.] If $\lambda<\lambda^\ast$, 
then the infimum $J^w_{p,\lambda}$ in (\ref{J2}) is not attained.
\item[3.] If $\lambda>\lambda^\ast$, 
then the infimum $J^w_{p,\lambda}$ in (\ref{J2}) is attained.
\end{enumerate}
\end{thm}
In particular we  have  the following inequality:
\begin{cor}\label{cor2.1}
Under the  same assumptions  as in Theorem \ref{Main}, 
there exists a constant $\lambda \in \R$ 
such  that for $u(x)\in W_{0}^{1,p}(\Omega;W_p(\delta))$
\begin{align}\label{HI1}
\int_\Omega & |\nabla u(x)|^p W_p(\delta(x))\,dx \nonumber \\
&\ge \Lambda_{p} \int_\Omega \frac{|u(x)|^p W_p(\delta(x))}{F_{\eta_0}(\delta(x))^p}\,dx 
+ \lambda \int_\Omega|u(x)|^p W_p(\delta(x))\,dx. 
\end{align}

\end{cor}
\begin{rem}\label{rem2.1}
\begin{enumerate}
\item For the case of $w(t)=1$ and $\lambda=0$, 
the value of the infimum $J^1_{p,0}$ in (\ref{J2}) and its attainability are studied in \cite{MMP}.
\item For the case of $w(t)=1$ and $p=2$, it is shown that 
the infimum $J_{2,\lambda}^1$ in (\ref{J2}) is attained if and only if $\lambda>\lambda^\ast$. 
See \cite{BM}.
If $ p\neq 2$ and  $\lambda=\lambda^\ast$, then it is  an open  problem whether the infimum $J^w_{p,\lambda}$ in (\ref{J2}) is achieved.
\item For the case of $w(t)=t^{\alpha p/(p-1)}\in Q_A(\R_+)$ with $\alpha<1-1/p$, 
Theorem \ref{Main} is shown in \cite{ah3}.
\item In the assertion 3 of Theorem \ref{Main}, 
the minimizer $u(x)\in W_{0}^{1,p}(\Omega;W_p(\delta))$ 
for the variational problem (\ref{J2}) is a non-trivial weak solution 
of the following Euler-Lagrange equation:
\begin{equation*}
-{\rm div}\left(W_p(\delta)|\nabla u|^{p-2}\nabla u\right)-\lambda W_p(\delta)|u|^{p-2}u
=J^w_{p,\lambda} \frac{W_p(\delta)}{F_{\eta_0}(\delta)^p} |u|^{p-2}u \quad \text{in} \ \ {\cal D}'(\Omega).
\end{equation*}
\end{enumerate}
\end{rem}
When $p=2$ and $ \lambda=\lambda^*$ hold, 
we have the following  that is rather precise.
\begin{thm}\label{p=2} In addition to the  assumption of Theorem \ref{Main}, we  assume  that
$p=2$ and $ \lambda=\lambda^*$. Let $\eta_0>0$ be a sufficiently small number as in Theorem \ref{CT3}. Moreover we assume  that
 \begin{equation}\lim_{t\to+0}F_{\eta_0}(t) G_{\eta_0}(t)^2=0 . \label{3.3}
 \end{equation}
Then, $J_{2,\lambda^*}^w$ is  not achieved.
\end{thm}
\begin{rem}\label{rem2.2} By  Theorem \ref{Main}, 
$J_{2,\lambda^*}^w=1/4$ holds.
\end{rem}
\begin{exam}
Let $w(t)=t^{\alpha p/(p-1)}$ for $\alpha\in\R$. Then $W_p(t)=t^{\alpha p}$. 
If $\alpha\ge 1-1/p$, then $w(t)\in P_A(\R_+)$, 
if $\alpha<1-1/p$, then $w(t)\in Q_A(\R_+)$.  
Clearly (\ref{mono}) is valid. 
We have that as $t\to+0$ 
$$\begin{array}{l}
F_{\eta_0}(t)=\begin{cases}O(t) & \text{for} \ \ \alpha\ne 1-1/p, 
\\[0.5ex] 
O\bigl(t\log(1/t)\bigr) & \text{for} \ \ \alpha=1-1/p,\end{cases} 
\vspace{2mm}\\
G_{\eta_0}(t)=\begin{cases}O\bigl(\log(1/t)\bigr) & \text{for} \ \ \alpha\ne 1-1/p, 
\\[0.5ex] 
O\bigl(\log\log(1/t)\bigr) & \text{for} \ \ \alpha=1-1/p.\end{cases}
\end{array}$$
Therefore (\ref{3.3}) holds.
\end{exam}
\begin{exam}{
Let either $w(t)=e^{-1/\sqrt{t}}\in P_A(\R_+)$ or $w(t)=e^{1/\sqrt{t}}\in Q_A(\R_+)$. \ 
Then (\ref{mono}) and (\ref{3.3}) hold. 
In fact, we have that as $t\to+0$ 
$$F_{\eta_0}(t)=O\bigl(t^{3/2}\bigr), \quad G_{\eta_0}(t)=O\bigl(t^{-1/2}\bigr), \quad 
F_{\eta_0}(t)G_{\eta_0}(t)^2=O\bigl(t^{1/2}\bigr).$$}
\end{exam}
\vspace{2mm}

Here we give the proof of Lemma \ref{lem lim F=0}.
\vspace{2mm}
\par
\noindent
{\bf Proof of Lemma \ref{lem lim F=0}.} \  
First we assume that $w(t)\in P(\R_+)$. 
Let $\varepsilon$ be any number satisfying $0<\varepsilon<2\eta_0$. 
For $0<t<\varepsilon/2$ we have that 
\begin{equation}\label{3.a}
F_{\eta_0}(t)=w(t)\left(\mu+\int_{\varepsilon/2}^{\eta_0}w(s)^{-1}ds\right)
+w(t)\int_t^{\varepsilon/2}w(s)^{-1}ds.
\end{equation}
Since $w(t)^{-1}\notin L^1((0,\eta_0))$, it follows that $\lim_{t\to+0}w(t)=0$ from 
the Definition \ref{D0}, and hence $w(t)$ is non-decreasing in $(0,\eta_0]$ by (\ref{mono}). 
Then we have 
\begin{equation}\label{3.b}
w(t)\int_t^{\varepsilon/2}w(s)^{-1}ds\le w(t)\int_t^{\varepsilon/2}w(t)^{-1}ds
=\frac{\varepsilon}{2}-t<\frac{\varepsilon}{2}.
\end{equation}
By $\lim_{t\to+0}w(t)=0$, there exists a $\delta>0$ such that for $0<t<\delta$ 
\begin{equation}\label{3.c}
w(t)<\frac{\varepsilon}{2\bigl(\mu+\int_{\varepsilon/2}^{\eta_0}w(s)^{-1}ds\bigr)}.
\end{equation}
From (\ref{3.a}), (\ref{3.b}) and (\ref{3.c}) it follows that for $0<t<\min\{\varepsilon/2,\delta\}$ 
\begin{equation*}
F_{\eta_0}(t)<\frac{\varepsilon}{2}+\frac{\varepsilon}{2}=\varepsilon,
\end{equation*}
which shows (\ref{lim F=0}). Secondly we assume that $w(t)\in Q(\R_+)$. 
If $w'(t)\ge 0$ for $t\in(0,\eta_0)$, then $\lim_{t\to+0}w(t)=a<\infty$, and so 
\begin{equation*}
F_{\eta_0}(t)=w(t)\int_0^t w(s)^{-1}ds \to 0 \quad\text{as} \ \ t\to+0
\end{equation*}
by $w(t)\in L^1((0,\eta_0))$. 
If $w'(t)\le 0$ for $t\in(0,\eta_0)$, then we see that for $t\in(0,\eta_0]$ 
\begin{equation*}
F_{\eta_0}(t)=w(t)\int_0^t w(s)^{-1}ds\le w(t)\int_0^t w(t)^{-1}ds=t,
\end{equation*}
which implies (\ref{lim F=0}). 
It concludes the proof.
\qed

\medskip
\section{Proof of Theorem {\ref{Main}}}
In this section, we give the proof of Theorem \ref{Main}.
\subsection{Upper bound of  $ J^w_{p,\lambda} $}
First, we prove the assertion 1 of Theorem \ref{Main}.
As test functions we adopt  
for $\varepsilon>0$ and $0< \eta \le \eta_0/2$
\begin{equation} u_\varepsilon (t)=
\begin{cases} f_{\eta_0}(t)^{1+s(w)\varepsilon -1/p}  &( 0<t\le \eta),\\
f_{\eta_0}(\eta)^{1+s(w)\varepsilon -1/p} (2\eta-t)/\eta &(\eta<t\le 2\eta),\\
 0  &(2\eta< t\le  \eta_0).\end{cases}
\end{equation}
We  note  that 
\begin{equation}  u'_\varepsilon(t)=\begin{cases}
 \left( 1 +s(w)\varepsilon -1/p\right) f_\eta(t)^{s(w)\varepsilon -1/p}
{s(w)}/{w(t)} &(0<t<\eta),\\
 - f_{\eta_0}(\eta)^{1+s(w)\varepsilon -1/p} /{\eta}  & (\eta< t< 2\eta),\\
 0 &(2\eta< t\le  \eta_0).
\end{cases}
\end{equation}
We  have 
\begin{align} \label{4.3}
\int_0^\eta|u_\varepsilon'(t)|^pW_p(t)\,dt
&= \left( 1-\frac 1p +s(w)\varepsilon\right)^{\!p}\int_0^\eta 
 f_{\eta_0}(t)^{s(w)\varepsilon p-1}\frac 1{w(t)} \,dt \nonumber \\
&= \left( 1-\frac 1p +s(w)\varepsilon\right)^{\!p} \frac{f_{\eta_0}(\eta)^{s(w) \varepsilon p}}{p\varepsilon}.
\end{align}
In  a similar way 
\begin{equation}
\int_0^\eta  \frac{|u_\varepsilon(t)|^p W_p(t)}{F_{\eta_0}(t)^p }\,dt
=\int_0^\eta  f_{\eta_0}(t)^{s(w)\varepsilon p-1}\frac 1{w(t)}\,dt
=  \frac{f_{\eta_0}(\eta)^{s(w) \varepsilon p}}{p\varepsilon}.
\end{equation}
Noting that $f_{\eta_0}(t)^{s(w)\e p}$  is bounded  by  the definitions of $s(w)$ and $f_{\eta_0}(t)$,  it follows from Lemma \ref{lem lim F=0} that 
\begin{align} \label{4.4}
\int_0^\eta  |u_\varepsilon(t)|^p W_p(t)\,dt
&=\int_0^\eta f_{\eta_0}(t)^{p-1+s(w)\e p}w(t)^{p-1}\,dt \nonumber \\ 
&=\int_0^\eta F_{\eta_0}(t)^{p-1}f_{\eta_0}(t)^{s(w)\e p}\,dt <+\infty.
\end{align}
Hence we have 
\begin{align*}
&\int_0^{2\eta}|u_\varepsilon'(t)|^pW_p(t)\,dt=
 \left( 1-\frac 1p +s(w)\varepsilon\right)^{\!p} \frac{f_{\eta_0}(\eta)^{s(w) \varepsilon p}}{p\varepsilon}+C(\e,\eta),\\
&\int_0^{2\eta } \frac{|u_\varepsilon(t)|^p W_p(t)}{F_{\eta_0} (t)^p}\,dt
=  \frac{f_{\eta_0}(\eta)^{s(w) \varepsilon p}}{p\varepsilon}+D(\e,\eta),\\
&\int_0^{2\eta } {|u_\varepsilon(t)|^p W_p(t)}\,dt=\int_0^\eta 
F_{\eta_0}(t)^{p-1}f_{\eta_0}(t)^{s(w)\e p}\,dt+ E(\e,\eta),
\end{align*}
where $C(\e,\eta)$, $D(\e,\eta)$ and  $E(\e,\eta)$ are given by
\begin{align*}
&C(\e,\eta) =  f_{\eta_0}(\eta)^{p+s(w)\varepsilon p -1}
\eta^{-p}\int_\eta^{2\eta}W_p(t)\,dt,\\
&D(\e,\eta)= f_{\eta_0}(\eta)^{p+s(w)\varepsilon p -1} 
\int_\eta^{2\eta}\frac{(2\eta-t)^p W_p(t)}{F_{\eta_0}(t)^p \eta^p}\,dt,\\
&E(\e,\eta)=f_{\eta_0}(\eta)^{p+s(w)\varepsilon p -1}\int_\eta^{2\eta}
\frac{(2\eta-t)^p W_p(t)}{\eta^p}\,dt,
\end{align*}
and they remain bounded as $\e\to+0$.
Therefore we see that
\begin{equation}\frac{\int_0^{2\eta }|u_\varepsilon'(t)|^pW_p(t) \,dt}
{\int_0^{2\eta }  |u_\varepsilon(t)|^p W_p(t) / F_{\eta_0}(t)^pdt}
\to \Lambda_p \quad \text{ as } \ \e\to +0,
\end{equation}
and we  also have 
\begin{equation}\frac{\int_0^{2\eta} |u_\varepsilon(t)|^p W_p(t) \,dt}
{\int_0^{2\eta } |u_\varepsilon(t)|^p W_p(t) / F_{\eta_0}(t)^pdt} \to 0 \quad \text{ as } \ \e\to +0.
\end{equation}
As a result we  have  the following lemma.
\begin{lem}\label{lem3.1}
Let $1<p<\infty$, $0<\eta\le \eta_0/2$  and $w(t)\in W(\mathbf R_+)$. 
For any $\kappa>0$, there exists a function $h(t)\in  W^{1,p}_0((0,2\eta);W_p)$ such  that
\begin{equation} 
\frac{ \int_0^{2\eta} |h'(t)|^p W_p(t)\,dt}{\int_0^{2\eta} |h(t)|^p W_p(t) / F_{\eta_0}(t)^pdt} 
\le \Lambda_p+\kappa. \label{upperbound}
\end{equation}
\end{lem}

\begin{proof}
By $ L^p((0,\eta); W_p)$ we denote the space of Lebesgue measurable functions with weight $ W_p(t)$, 
for which 
\begin{equation*} 
\| u \|_{ L^p((0,\eta); W_p)} = \bigg( \int_0^\eta|u(t)|^p W_p(t)\,dt\bigg ) ^{\!1/p} < +\infty. 
\end{equation*} 
 $W_{0}^{1,p}((0,\eta); W_p)$ is given by the  completion of $C_c^\infty((0,\eta))$ with respect to the norm defined by
\begin{equation*} 
\| u\|_{ W_{0}^{1,p}((0,\eta); W_p)} =
\|u' \|_{ L^p((0,\eta); W_p)} + \| u\|_{ L^p((0,\eta); W_p)}. 
\end{equation*}
Then 
 $W_{0}^{1,p}((0,\eta); W_p) $  becomes a Banach space  
with the norm $\| \cdot\|_{ W_{0}^{1,p}((0,\eta); W_p)}$. 

Let  us  set
  $ h(t)= u_\e(t)$ for a sufficiently small $\e>0$.  Then $h(t)$ satisfies the  estimate (\ref{upperbound}).
It suffices to  check that $h(t)\in  W^{1,p}_0((0,2\eta);W_p)$. 
If $w(t)\in Q(\mathbf R_+)$,  then  $ \lim_{t\to +0}f_{\eta_0}(t)=0$ and   
  $ \lim_{t\to +0}u_\e(t) = \lim_{t\to +0}f_{\eta_0}(t)^{1+s(w)\varepsilon -1/p}=0$. Therefore $h(t)$ is clearly  approximated by test functions in $C^\infty_c((0,2\eta))$. 
  \par\medskip
   If $w(t)\in P(\mathbf R_+)$, then  we employ the following lemma:
 \begin{lem}\label{approximation} Assume  that $1<p<\infty$ and $w(t)\in P(\mathbf R_+)$. 
 For $\e>0$, $\eta>0$  and $\eta_0>0$  satisfying $0<\eta\le \eta_0/2$, let us  set 
 \begin{equation} \varphi_\varepsilon(t)= \,\,
0 \,\,( 0\le t\le \varepsilon) ;  \quad \frac {f_{\eta_0}(\varepsilon)-f_{\eta_0}(t)  }{f_{\eta_0}(\varepsilon)-f_{\eta_0}(\eta) }\, \,(\varepsilon \le t\le \eta);\quad  1\, \,(\eta \le t \le 2\eta). \label{7.1}\end{equation}
Then, as $\e$  
$\to +0$, 
$\varphi_\e \to 1 $ in $L^p((0,2\eta);W_p)$ and $\varphi'_\e\to0$ in $L^p((0,2\eta);W_p)$.
\end{lem}

\begin{proof}
Since $\lim_{t\to +0}f_{\eta_0}(t)=\infty$, 
clearly  $\varphi_\e(t)\to 1 $ 
{
in $L^p((0,2\eta);W_p)$ as $\e\to +0$, 
} 
and 
$ \int _0^{2\eta}|\varphi_\varepsilon '(t)|^p W_p(t)\,dt
=( f_{\eta_0}(\varepsilon)-f_{\eta_0}(\eta))^{1-p} \to 0$ as $\varepsilon\to +0$. 
Then we  see  the assertion.
\end{proof}

\par\medskip\noindent{\bf {End of  the proof of Lemma \ref{lem3.1}:}} \ 
For  $0<\overline\e<\eta$,
we  set  $h_{\overline \e}(t)=\varphi_{\overline \e}(t) h(t)$, where $\varphi_{\overline \e}(t)$ is defined by (\ref{7.1}) with $\e=\overline \e $. 
Then $ \supp h_{\overline \e}(t)\subset\, [ \overline \e, 2\eta]$.
By  virtue of Lemma \ref{approximation}, we  also see that 
 $h_{\overline \e}(t)\to h(t)$ in $W^{1,p}((0,2\eta);W_p)$ as $\overline \e \to+0$.
In fact, noting  that $ h'_{\overline \e}(t)= \varphi'_{\overline \e}(t) h(t)+\varphi_{\overline \e}(t) h' (t)$, we  have
\begin{equation*}
\begin{split}
\int_0^{2\eta } &|  h'_{\overline\e}(t)- h'(t)|^pW_p(t)\,dt \\
&\le C_p\left(\int_0^{2\eta } (1-\varphi_{\overline \e }(t) )^p| h'(t)|^pW_p(t)\,dt +
\int_0^{2\eta } |\varphi'_{\overline \e }(t) |^p |h(t)|^p W_p(t)\,dt\right)
\end{split}
\end{equation*}
with some constant $C_p>0$ depending only on $p$. 
The first term obviously goes  to $0$  as $\overline\e\to +0$.  
As for the second,  noting  that $s(w)=-1$  and $0<\e<1$, we have 
\begin{align*}\int_0^{2\eta } |\varphi'_{\overline \e }(t) |^p |h(t)|^pW_p(t)\,dt&=
\int_{\overline \e }^{\eta} |\varphi'_{\overline \e }(t) |^p |h(t)|^pW_p(t)\,dt\\
&=\frac1{( f_{\eta_0}(\overline \e )- f_{\eta_0}(\eta))^{p}}\int_{\overline \e }^{\eta}\frac{f_{\eta_0}(t)^{p-1+ps(w)\e }}{w(t)}\,dt\\
&= \frac{1}{p(1-\e)} \frac{ f_{\eta_0}(\overline \e)^{p(1-\e)}- f_{\eta_0}(\eta)^{p(1-\e)}}{( f_{\eta_0}(\overline \e )- f_{\eta_0}(\eta))^{p}}.
\end{align*}
Since $\lim_{t\to+0} f_{\eta_0}(t)=\infty$, 
we see that $\int_0^{2\eta } |\varphi'_{\overline \e} (t)|^p |h(t)|^pW_p(t)\,dt\to 0$ as $\overline \e\to +0$.
Since $h_{\overline{\e}}(t)$ is clearly  approximated 
by test functions in $C^\infty_c((0,2\eta))$, 
 the assertion $h(t)\in W_0^{1,p}((0,2\eta);W_p)$ follows.
\end{proof}
 
\begin{lem}\label{lem3.2}
Let $\Omega$ be  a  bounded domain of  class $C^2$ in ${\R}^N$. 
Let $1<p<\infty$ and $w(t)\in W(\mathbf R_+)$. 
Then it holds that 
\begin{equation}\label{ubd} 
J^w_{p,\lambda} \le \Lambda_{p} 
\end{equation}
for all $\lambda\in\R$.
\end{lem}
\begin{proof}
For each  small $\eta>0$, 
by $\Omega_\eta $ we denote a tubular neighborhood of $\partial \Omega$;
\begin{equation}\label{t-nbd}
\Omega_\eta = \{ x\in \Omega:  \delta(x)= {\rm dist}(x,\partial\Omega)<\eta \}. 
\end{equation}
Since the boundary $\partial\Omega$ is of class $C^2$, there exists an $\eta_0>0$ such that 
for any $\eta\in(0,\eta_0)$ and every $x\in\Omega_\eta$ 
we have a unique point $\sigma(x)\in\partial\Omega$ satisfying $\delta(x)=|x-\sigma(x)|$. 
The mapping 
\begin{equation*}
\Omega_\eta\ni x\mapsto (\delta(x),\sigma(x))=(t,\sigma)\in(0,\eta)\times\partial\Omega
\end{equation*}
is a $C^2$ diffeomorphism, and its inverse is given by 
\begin{equation*}
(0,\eta)\times\partial\Omega\ni(t,\sigma)\mapsto 
x(t,\sigma)=\sigma+t\cdot n(\sigma)\in\Omega_\eta,
\end{equation*}
where $n(\sigma)$ is the inward unit normal to $\partial\Omega$ at $\sigma\in\partial\Omega$. 
For each $t\in(0,\eta)$, the mapping 
\begin{equation*}
\partial\Omega\ni\sigma\mapsto 
\sigma_t(\sigma)=x(t,\sigma)\in\Sigma_t=\{x\in\Omega:\delta(x)=t\}
\end{equation*}
is also a $C^2$ diffeomorphism of $\partial\Omega$ onto $\Sigma_t$, 
and its Jacobian satisfies 
\begin{equation}\label{3.5}
|\text{Jac}\,\sigma_t(\sigma)-1|\le ct \qquad \text{for any} \ \ \sigma\in\partial\Omega,
\end{equation}
where $c$ is a positive constant depending only on $\eta_0$, $\partial\Omega$  
and the choice of local coordinates. 
Since $n(\sigma)$ is orthogonal to $\Sigma_t$ 
at $\sigma_t(\sigma)=\sigma+t\cdot n(\sigma)\in\Sigma_t$, 
it follows that for every integrable function $v(x)$ in $\Omega_\eta$
\begin{align}\label{3.6}
\int_{\Omega_\eta}v(x)dx 
& =\int_0^\eta dt\int_{\Sigma_t}v(\sigma_t)d\sigma_t \nonumber \\
& =\int_0^\eta dt\int_{\partial\Omega}v(x(t,\sigma))|\text{Jac}\,\sigma_t(\sigma)|d\sigma,
\end{align}
where $d\sigma$ and $d\sigma_t$ denote surface elements on 
$\partial\Omega$ and $\Sigma_t$, respectively. 
Hence (\ref{3.6}) together with (\ref{3.5}) implies that 
for every integrable function $v(x)$ in $\Omega_\eta$
\begin{align}
\int_0^\eta (1-ct)dt\int_{\partial\Omega}|v(x(t,\sigma))|d\sigma 
& \le \int_{\Omega_\eta}|v(x)|dx \label{3.7} 
\\ 
& \le \int_0^\eta (1+ct)dt\int_{\partial\Omega}|v(x(t,\sigma))|d\sigma. \label{3.8}
\end{align}

Let $\kappa>0$, and let $\eta\in(0,\eta_0)$. 
Take $h(t)\in  W^{1,p}_0((0,\eta);W_p)$  be a function satisfying (\ref{upperbound}) with replacing $2\eta$ by $\eta$ for  simplicity.
Define
\begin{equation}\label{3.9}
u(x)= \begin{cases} h(\delta(x)) & \text{if} \quad x\in \Omega_\eta,\\ 
  0 & \text{if} \quad x\in \Omega\setminus\Omega_\eta. \end{cases}  
\end{equation}
Then we  have $\supp u\subset \Omega_\eta$. 
Since $|\nabla u(x)|=|h'(\delta(x))|$ for $x\in\Omega_\eta$ by $|\nabla \delta(x)|=1$, 
it follows from (\ref{3.8}) that 
\begin{equation}\label{3.10}
\int_{\Omega_\eta}|\nabla u(x)|^pW_p(\delta(x))\, dx 
\le (1+c\eta)|\partial\Omega|\int_0^\eta |h'(t)|^pW_p(t)\, dt,  
\end{equation}
which implies $u(x)\in W_{0}^{1,p}(\Omega;W_p(\delta))$ by Lemma \ref{lem3.1}.
On the other hand, by (\ref{3.7}) and (\ref{3.9}) we have that 
\begin{equation}\label{3.11}
\int_{\Omega_\eta}|u(x)|^p\frac{W_p(\delta(x))}{F_{\eta_0}(\delta(x))^p}dx 
\ge (1-c\eta)|\partial\Omega|\int_0^\eta |h(t)|^p\frac{W_p(t)}{F_{\eta_0}(t)^p}dt.  
\end{equation}
By combining (\ref{3.10}), (\ref{3.11}) and  trivial estimate 
\begin{align}&
\int_{\Omega_\eta}|u(x)|^pW_p(\delta(x))\,dx 
\le \Bigl(\,\sup_{0<t<\eta}F_{\eta_0}(t)\Bigr)^p
\int_{\Omega_\eta}|u(x)|^p\frac{W_p(\delta(x))}{F_{\eta_0}(\delta(x))^p}\,dx, 
\end{align}
we obtain that 
\begin{equation*}
 \chi^w_{p,\lambda} (u)
\le \frac{1+c\eta}{1-c\eta}\frac{\int_0^\eta |h'(t)|^pW_p(t)\,dt}
{\int_0^\eta |h(t)|^p W_p(t) / F_{\eta_0}(t)^pdt} 
+|\lambda|\Bigl(\,\sup_{0<t<\eta}F_{\eta_0}(t)\Bigr)^p. 
\end{equation*}
This together with Lemma \ref{lem3.1} implies that 
\begin{equation}\label{3.12}
J_{p,\lambda}^w
\le \frac{1+c\eta}{1-c\eta}(\Lambda_{p}+\kappa) 
+|\lambda|\Bigl(\,\sup_{0<t<\eta}F_{\eta_0}(t)\Bigr)^p. 
\end{equation}
Letting $\eta\to +0$ and $\kappa\to+0$ in (\ref{3.12}),  
then (\ref{ubd}) follows from Lemma \ref{lem lim F=0}. 
Therefore it concludes the proof.
\end{proof}

\begin{lem}\label{lem3.3}
Let $\Omega$ be  a  bounded domain of  class $C^2$ in ${\R}^N$. 
Let $1<p<\infty$ and $w(t)\in W_A({\R}_+)$. 
Then there exists a $\lambda\in\R$ such that $J^w_{p,\lambda} =\Lambda_{p}$.
\end{lem}
 
\begin{proof}
Let $\eta_0>0$ be a sufficiently small number as in Theorem \ref{CT3}. 
Take  and fix any $u(x)\in W_{0}^{1,p}(\Omega;W_p(\delta))\setminus\{0\}$. 
Then, for $\eta\in (0,\eta_0]$ 
\begin{align}
\int_\Omega &\frac{|u(x)|^p W_p(\delta(x))}{F_{\eta_0}(\delta(x))^p}\,dx \nonumber \\
& =\int_{\Omega_\eta} \frac{|u(x)|^p W_p(\delta(x))}{F_{\eta_0}(\delta(x))^p} \,dx
+\int_{\Omega\setminus \Omega_\eta} \frac{|u(x)|^p W_p(\delta(x))}{F_{\eta_0}(\delta(x))^p}\,dx.
\end{align}
Since there exists a positive  number  $C_\eta$ independent of  $u(x)$ such  that 
\begin{equation}\label{4.24}
\int_{\Omega\setminus\Omega_\eta}|u(x)|^p\frac{W_p(\delta(x))}{F_{\eta_0}(\delta(x))^p}\,dx
\le C_\eta\int_{\Omega\setminus\Omega_\eta}|u(x)|^pW_p(\delta(x))\,dx,
\end{equation}
by using Hardy's inequality (\ref{2.11}) we have 
\begin{align}\label{4.26}
\Lambda_{p} &\int_{\Omega} \frac{|u(x)|^p W_p(\delta(x))}{F_{\eta_0}(\delta(x))^p}\,dx\le
\int_{\Omega_\eta}|\nabla u(x)|^pW_p(\delta(x))\,dx \nonumber \\ 
&-  s(w)L' \int_{\Sigma_\eta} |u(\sigma_\eta)|^p W_p(\eta)\,d\sigma_\eta
+\Lambda_{p}C_\eta\int_{\Omega\setminus\Omega_\eta}|u(x)|^pW_p(\delta(x))\,dx.
\end{align}
In order to control  the  integrand on there surface $\Sigma_\eta$ we prepare the following:
\begin{lem}\label{lemtrace}
Let $\Omega$ be  a  bounded domain of  class $C^2$ in ${\R}^N$. 
Let $1<p<\infty$ and $w(t)\in W(\mathbf R_+)$. 
Assume that $\eta_0 $ is a  sufficiently small positive  number and $\eta\in (0,\eta_0/3)$. Then, for  any $\e>0$ there exists a positive  number $C_{\varepsilon,\eta}$ 
such  that we have for any  $u(x)\in W^{1,p}_0(\Omega; W_p(\delta))$
\begin{equation}
\| u\|^p_{ L^p( \Sigma_\eta; W_p(\delta))}\le \e 
\| |\nabla u | \|^p_{{ L^p( \Omega_{3\eta}\setminus \overline{\Omega_\eta}; W_p(\delta))}}+ 
C_{\varepsilon,\eta}  
\| u\|^p_{ L^p( \Omega_{3\eta}\setminus \overline{\Omega_\eta}; W_p(\delta))}. \label{traceineq}
\end{equation}
Here we  denote by $\overline{\Omega_\eta}$ the  closure of ${\Omega_\eta}$.
\end{lem}
\begin{rem}\label{rem4.1}By Rellich's theorem and Hardy type inequality, 
we see that the imbedding 
$W^{1,p}_0(\Omega;W_p(\delta)) \hookrightarrow L^p(\Omega; W_p(\delta))$ is compact.
Therefore, by this lemma we  see that 
{a trace operator}
$ W^{1,p}_0(\Omega; W_p(\delta)) 
\to L^p(\Sigma_\eta ; W_p(\delta))$ 
is also compact. 
\end{rem}

\begin{proof}
For $\eta\in (0, \max_{x\in \Omega} \delta(x)/3)$,
let  $W^{1,p}(\Omega_{3\eta}\setminus \overline{\Omega_\eta}; W_p(\delta))$  be 
 given by the  completion of $C^\infty(\Omega_{3\eta}\setminus \overline{\Omega_\eta})$ with respect to the norm 
defined by
\begin{equation*} 
\| u\|_{W^{1,p}(\Omega_{3\eta}\setminus \overline{\Omega_\eta}; W_p(\delta)) } =
\| |\nabla u|  \|_{ L^p(\Omega_{3\eta}\setminus \overline{\Omega_\eta}; W_p(\delta))} + \| u\|_{ L^p(\Omega_{3\eta}\setminus \overline{\Omega_\eta}; W_p(\delta))}. 
\end{equation*}
Since $W_p(\delta(x))>0$ in $\overline{\Omega_{3\eta}\setminus\overline{\Omega_\eta}}$, $W^{1,p}(\Omega_{3\eta}\setminus \overline{\Omega_\eta}; W_p(\delta)) $  is well-defined and  becomes a Banach space  with the norm 
$\| \cdot\|_{W^{1,p}(\Omega_{3\eta}\setminus \overline{\Omega_\eta}; W_p(\delta)) }$.

Hence the inequality (\ref{traceineq}) follows from the standard theory for a trace operator 
$$W^{1,p}(\Omega_{3\eta}\setminus \overline{\Omega_\eta}; W_p(\delta)) 
\to   L^p(\Sigma_\eta; W_p(\delta)).$$ Here we give a simple proof of  it.
We use the following cut-off function $ \psi(x)\in C^\infty( \Omega)$ such  that
$ \psi(x)\ge 0$  and 
\begin{equation}\psi(x)= \begin{cases}
 1 &  (x\in \Omega_{2\eta}),\\
 0 &  (x\in \Omega\setminus \Omega_{3\eta}).
\end{cases}
\end{equation}
We retain the  notations in the proof of Lemma \ref{lem3.2}. 
Take and fix a  $u(x)\in W^{1,p}_0(\Omega;W_p(\delta))$ and  assume $u(x)\ge 0$.  Then,
\begin{align*}\int_{\Sigma_\eta}& u(\sigma_\eta)^p  W_p(\eta)\,d\sigma_\eta
=\int_{\partial\Omega}
u(x(\eta,\sigma))^pW_p(\eta)|\text{Jac}\,\sigma_\eta(\sigma)|\,d\sigma\\
=& -\int_{\partial\Omega}d\sigma \int_\eta^{3\eta} \frac{\partial}{\partial t}\left(
u(x(t,\sigma))^p \psi (x(t,\sigma)) W_p(t)|\mbox{Jac}\,\sigma_t(\sigma)|\right)\,dt\\
=&- \int_{\partial\Omega} \,d\sigma 
\int_\eta^{3\eta} \frac{\partial}{\partial t}\left(
u(x(t,\sigma))^p \right)\cdot \psi (x(t,\sigma)) W_p(t)|
\mbox{Jac}\,\sigma_t(\sigma)|\,dt \\
&-\int_{\partial\Omega} \,d\sigma 
\int_\eta^{3\eta} 
u(x(t,\sigma))^p \cdot \frac{\partial}{\partial t}\left(\psi (x(t,\sigma)) W_p(t)|\mbox{Jac}\,\sigma_t(\sigma)| \right)\,dt\\
=& \ I_1+ I_2.
\end{align*}
Note  that $x(t,\sigma), W_p(t),\mbox{Jac}\,\sigma_t(\sigma) \in C^1$ in $t\in (\eta,3\eta)$ and 
$$ \int_{\Omega_{3\eta}\setminus \Omega_\eta} u(x)^p\,dx=\int_{\partial\Omega}\,d\sigma
\int_\eta^{3\eta} u(x(t,\sigma))^p |\mbox{Jac}\,\sigma_t(\sigma)| \,dt.
$$
Then,  we have  for  some $C_\eta>0$ independent of $u(x)$  
$$ |I_2| \le C_\eta
\int_{\Omega_{3\eta}\setminus \Omega_\eta}u(x)^p W_p(\delta(x))\,dx.$$
As for $I_1$, for any $\e>0$ there is a positive number $C_\e$ 
independent of $u(x)$ and $\eta$ such  that we  have 
\begin{align*} |I_1|&\le \e \int_{\Omega_{3\eta}\setminus \Omega_\eta}|\nabla u(x)|^pW_p(\delta(x))\,dx+
C_\e \int_{\Omega_{3\eta}\setminus \Omega_\eta}u(x)^p W_p(\delta(x))\,dx.
\end{align*}
Therefore we obtain (\ref{traceineq}). It concludes the proof of Lemma \ref{lemtrace}.
\end{proof}

\par\medskip

\noindent{\bf {End of  the proof of Lemma \ref{lem3.3}:}} \ 
From (\ref{4.26}) and Lemma \ref{lemtrace}, it follows that 
\begin{align*}
& \int_{\Omega_\eta}|\nabla u(x)|^pW_p (\delta(x))\,dx \\ & \ge 
\Lambda_{p}\int_{\Omega} \frac{|u(x)|^p W_p(\delta(x))}{F_{\eta_0}(\delta)^p}\,dx 
-\Lambda_{p}C_\eta
\int_{\Omega\setminus\Omega_\eta}|u(x)|^pW_p(\delta(x))\,dx\\
& \quad 
-L' \left(\e\int_{\Omega_{3\eta}\setminus \Omega_\eta}|\nabla u(x)|^pW_p(\delta(x))\,dx
+C_{\varepsilon,\eta}
\int_{\Omega_{3\eta}\setminus \Omega_\eta}|u(x)|^p W_p(\delta(x))\,dx \right),
\end{align*}
and so 
\begin{align*}
 \int_{\Omega_\eta} & |\nabla u(x)|^pW_p(\delta(x)) \,dx+ L'\e \int_{\Omega_{3\eta}\setminus \Omega_\eta}|\nabla u(x)|^pW_p(\delta(x))\,dx \\  \ge & \ 
\Lambda_{p}\int_{\Omega} \frac{|u(x)|^p W_p(\delta(x))}{F_{\eta_0}(\delta(x))^p} \,dx
-\bigl(L' C_{\varepsilon,\eta}+\Lambda_{p}C_\eta\bigr)
\int_{\Omega\setminus \Omega_\eta}|u(x)|^p W_p(\delta(x))\,dx.
\end{align*}
Now  we  set $L'\e=1$ and $C'= -\bigl(L' C_{\varepsilon,\eta} 
+\Lambda_{p}C_\eta\bigr)<0$, and  we  have  the desired estimate: 
\begin{align*}
 \int_{\Omega}|\nabla u(x)|^pW_p(\delta(x))\,dx \ge 
\Lambda_{p}\int_{\Omega} \frac{|u(x)|^p W_p(\delta(x))}{F_{\eta_0}(\delta(x))^p} \,dx
+C'\int_{\Omega}|u(x)|^p W_p(\delta(x))\,dx,
\end{align*}
which implies that
\begin{equation*}
 \chi^w_{p,\lambda} (u)\ge \Lambda_{p}
\end{equation*}
for $\lambda\le C'$. 
Consequently, it holds that $J^w_\lambda \ge \Lambda_{p}$ for $\lambda\le C'$.  
This together with (\ref{ubd}) implies the desired conclusion. 
It completes the proof of Lemma \ref{lem3.3}.
\end{proof}

\noindent
{\bf  Proof of the assertion 1 of Theorem \ref{Main}.} \ 
By Lemma \ref{lem3.3} and $\lim_{\lambda\to\infty}J_{p,\lambda}^w=-\infty$,  
the set $\{\lambda \in \mathbf R : J_{p,\lambda}^w = \Lambda_{p} \}$ 
is non-empty and upper bounded. 
Hence the 
$\sup\{\lambda \in \mathbf R : J_{p,\lambda}^w = \Lambda_{p} \}$ exists finitely. 
Put 
\begin{equation}
\lambda^\ast = \sup\{\lambda \in \mathbf R : J_{p,\lambda}^w = \Lambda_{p} \}. \label{lambda*}
\end{equation}
Since the function $\lambda\mapsto J_{p,\lambda}^w$ is non-increasing on $\mathbf R$, 
it follows from Lemma \ref{lem3.2} and Lemma \ref{lem3.3} that 
$J_{p,\lambda}^w=\Lambda_{p}$ for $\lambda<\lambda^\ast$ 
and $J_{p,\lambda}^w<\Lambda_{p}$ for $\lambda>\lambda^\ast$. 
Since $J_{p,\lambda}^w$ is  clearly  Lipschitz continuous on $\mathbf R$ with respect  to $\lambda$,
we have the equality $J_{p,\lambda^\ast}^w=\Lambda_{p}$.
Therefore the assertion 1 of Theorem \ref{Main} is valid.
\qed

\subsection{ $J_{p,\lambda}^w$ is  not  attained when $\lambda < \lambda^\ast$}

Next, we prove the assertion 2 of Theorem \ref{Main}.

\medskip

\noindent
{\bf  Proof of the assertion 2 of Theorem \ref{Main}.} \ 
Suppose that for some $\lambda<\lambda^\ast$ 
the infimum $J_{p,\lambda}^w$ in (\ref{J2}) is attained 
at an element $u\in W^{1,p}_{0}(\Omega;W_p(\delta))\setminus\{0\}$. 
Then, by the assertion 1 of Theorem \ref{Main}, we have that 
\begin{equation}\label{3.14}
 \chi^w_{p,\lambda}(u) =J_{p,\lambda}^w= \Lambda_{p}
\end{equation}
and for $\lambda<\bar{\lambda}<\lambda^\ast$
\begin{equation}\label{3.15}
\chi_{p,\bar{\lambda}}^w (u) \ge J_{p,\bar{\lambda}}^w= \Lambda_{p}. 
\end{equation}
From (\ref{3.14}) and (\ref{3.15}) it follows that 
\begin{equation*}
(\bar{\lambda}-\lambda)\int_{\Omega}|u(x)|^pW_p(\delta(x))\,dx \le 0. 
\end{equation*}
Since $\bar{\lambda}-\lambda>0$, we conclude that 
\begin{equation*}
\int_{\Omega}|u(x)|^pW_p(\delta(x))\,dx = 0,
\end{equation*}
which contradicts $u\ne 0$ in $W_{0}^{1,p}(\Omega;W_p(\delta))$. 
Therefore it completes the proof.
\qed

\subsection{Attainability of $J^w_{p,\lambda}$ when $\lambda>\lambda^\ast$ }

At last, we prove the assertion 3 of Theorem \ref{Main}. 
{
Let $\eta_0$ be  sufficiently small as in Theorem \ref{CT3} and let $\eta\in (0,\eta_0]$. 
}
Let $\{u_k\}$ be a minimizing sequence for the variational problem (\ref{J2}) normalized so that 
\begin{equation}\label{ms1}
\int_{\Omega} \frac{|u_k(x)|^p W_p(\delta(x))}{F_{\eta_0}(\delta(x))^p}  \,dx= 1 \quad \text{for all} \ k.  
\end{equation}
Since $\{u_k\}$ is bounded in $W_{0}^{1,p}(\Omega;W_p(\delta))$, by taking a suitable subsequence, 
we may assume that there exists a $u\in W_{0}^{1,p}(\Omega;W_p(\delta))$ such that 
\begin{align}
\nabla u_k \stackrel{weak}{\longrightarrow} \nabla u \quad 
& \text{in} \ \  (L^p(\Omega;W_p(\delta)))^N, \label{ms2} \\
u_k  \stackrel{weak}{\longrightarrow} u \quad 
& \text{in} \ \  L^p(\Omega;W_p(\delta)/ F_{\eta_0}(\delta)^p), \label{ms3} \\ 
u_k \longrightarrow u \quad & \text{in} \ \  L^p(\Omega;W_p(\delta)) \label{ms4}
\end{align}
and
\begin{equation}\label{} 
u_k \longrightarrow u \quad \text{in} \ \  L^p(\Sigma_\eta; W_p(\delta)) \label{ms5}
\end{equation}
by Remark \ref{rem4.1}. 
Under these preparation we establish the properties of concentration and compactness 
for the minimizing sequence, respectively.

\begin{prop}\label{concentration}
Let $\Omega$ be  a  bounded domain of  class $C^2$ in ${\R}^N$. 
Let $1<p<\infty$ and $w(t)\in W_A(\mathbf R_+)$. 
Let $\lambda\in  \R$. 
Let $\{u_k\}$ be a minimizing sequence for (\ref{J2}) satisfying (\ref{ms1}) 
and $(\ref{ms2})\sim(\ref{ms5})$ with $u=0$. 
Then it holds that 
\begin{equation}\label{3.20}
\nabla u_k \longrightarrow 0 \quad \text{in} \ \ (L^p_{\rm loc}(\Omega;W_p(\delta)))^N 
\end{equation}
and 
\begin{equation}\label{3.21}
J_{p,\lambda}^w = \Lambda_{p}.  
\end{equation}
\end{prop}

\begin{proof} 
Let $\eta_0>0$ be a sufficiently small number as in Theorem \ref{CT3} and  let $\eta\in (0,\eta_0]$. 
By Hardy's inequality (\ref{2.11}) and (\ref{ms1}) we have that 
\begin{align*}
\int_{\Omega_\eta} &|\nabla u_k(x)|^p W_p(\delta(x))\,dx \\ &
\ge
\Lambda_{p}\!\int_{\Omega_\eta}\frac{|u_k(x)|^pW_p(\delta(x))}{F_{\eta_0}(\delta(x))^p}\,dx 
+ s(w)L' \int_{\Sigma_\eta} |u_k(\sigma_\eta)|^p W_p(\eta)\,d\sigma_\eta 
\nonumber
\\ 
& =\Lambda_{p}\!
\left(1-\int_{\Omega\setminus\Omega_\eta}
\frac{|u_k(x)|^pW_p(\delta(x))}{F_{\eta_0}(\delta(x))^p}dx\right)
+ s(w)L' \int_{\Sigma_\eta} |u_k(\sigma_\eta)|^p W_p(\eta)\,d\sigma_\eta, \nonumber 
\end{align*}
and so 
\begin{align}\label{3.22}
\chi_\lambda^\alpha(u_k) 
&\ge \Lambda_{p}\!
\left(1-\int_{\Omega\setminus\Omega_\eta}\frac{|u_k(x)|^pW_p(\delta(x))}{F_{\eta_0}(\delta(x))^p}\,dx\right) 
+ s(w)L' \int_{\Sigma_\eta} |u_k(\sigma_\eta)|^p W_p(\eta)\,d\sigma_\eta \nonumber \\ 
& \quad + \int_{\Omega\setminus\Omega_\eta}|\nabla u_k(x)|^pW_p(\delta(x))\,dx 
-\lambda \int_\Omega |u_k(x)|^pW_p(\delta(x))\,dx. 
\end{align}
Since there exists a positive number $C_\eta$ independent of $u_k$ such that 
\begin{equation*}
\int_{\Omega\setminus\Omega_\eta}\frac{|u_k(x)|^pW_p(\delta(x))}{F_{\eta_0}(\delta(x))^p}dx 
\le C_\eta \int_{\Omega}|u_k(x)|^pW_p(\delta(x))\,dx,
\end{equation*}
it follows from (\ref{ms4}) with $u=0$ that 
\begin{equation}\label{3.23}
\lim_{k\to\infty}\int_{\Omega\setminus\Omega_\eta}\frac{|u_k(x)|^pW_p(\delta(x))}{F_{\eta_0}(\delta(x))^p}\,dx=0  
\end{equation}
Hence, letting $k\to\infty$ in (\ref{3.22}), 
by (\ref{3.23}), (\ref{ms4}) and (\ref{ms5}) with $u=0$, we obtain that 
\begin{equation*}
0\le\limsup_{k\to\infty}\int_{\Omega\setminus\Omega_\eta}|\nabla u_k(x)|^pW_p(\delta(x))\,dx 
\le J_{p,\lambda}^w -\Lambda_{p}. 
\end{equation*}
Since $J_{p,\lambda}^w -\Lambda_{p}\le 0$ by Lemma \ref{lem3.2}, we conclude that 
$J_{p,\lambda}^w -\Lambda_{p}=0$ and 
\begin{equation}\label{3.24}
\lim_{k\to\infty}\int_{\Omega\setminus\Omega_\eta}|\nabla u_k(x)|^pW_p(\delta(x))\,dx=0. 
\end{equation}
These show (\ref{3.20}) and (\ref{3.21}). 
Consequently it completes the proof.
\end{proof}

\begin{prop}\label{compactness}
Let $\Omega$ be  a  bounded domain of  class $C^2$ in ${\R}^N$. 
Let $1<p<\infty$, $w(t)\in W_A(\mathbf R_+)$ and $\lambda\in  \R$. 
Let $\{u_k\}$ be a minimizing sequence for (\ref{J2}) satisfying (\ref{ms1}) 
and $(\ref{ms2})\sim(\ref{ms5})$ with $u\ne 0$. 
Then it holds that 
\begin{equation}\label{3.25}
J_{p,\lambda}^w =\min\bigl(\Lambda_{p},  \chi^w_{p,\lambda}(u)\bigr). 
\end{equation}
In addition, if $J_{p,\lambda}^w <\Lambda_{p}$, then it holds that 
\begin{equation}\label{minimizer} 
J_{p,\lambda}^w= \chi^w_{p,\lambda}(u), 
\end{equation}
namely $u$ is a minimizer for (\ref{J2}), and  
\begin{equation}\label{3.27}
u_k \longrightarrow u \quad \text{in} \ \ W_{0}^{1,p}(\Omega;W_p(\delta)). 
\end{equation} 
\end{prop}

\begin{proof} 
Let $\eta_0>0$ be a sufficiently small number as in Theorem \ref{CT3} and  let $\eta\in (0,\eta_0]$. 
Then we have (\ref{3.22}) by the same arguments as in the proof of Proposition 
\ref{concentration}. 
Since there exists a positive number $C_\eta$ independent of $u_k$ such that 
\begin{equation*}
\int_{\Omega\setminus\Omega_\eta}\frac{|u_k(x)-u(x)|^pW_p(\delta(x))}{F_{\eta_0}(\delta(x))^p} \,dx 
\le C_\eta \int_{\Omega}|u_k(x)-u(x)|^pW_p(\delta(x))\,dx,
\end{equation*}
(\ref{ms4}) implies that 
\begin{equation}\label{3.28} 
\lim_{k\to\infty}\int_{\Omega\setminus\Omega_\eta}\frac{|u_k(x)|^pW_p(\delta(x))}{F_{\eta_0}(\delta(x))^p}\,dx 
= \int_{\Omega\setminus\Omega_\eta}\frac{|u(x)|^pW_p(\delta(x))}{F_{\eta_0}(\delta(x))^p}\,dx. 
\end{equation}
Since it follows from (\ref{ms2}) that $\nabla u_k \longrightarrow \nabla u$ 
weakly in $(L^p(\Omega\setminus \overline{\Omega_\eta};W_p(\delta)))^N$, 
by weakly lower semi-continuity of the $L^p$-norm, we see that 
\begin{align}\label{3.29} 
\liminf_{k\to\infty} \int_{\Omega\setminus {\Omega_\eta}}|\nabla u_k(x)|^pW_p(\delta(x))\,dx 
& \ge \left(\liminf_{k\to\infty} 
\| |\nabla u_k| \|_{L^p(\Omega\setminus\overline{\Omega_\eta} ;W_p(\delta))}\right)^{\!p} 
\nonumber \\ 
& \ge \| |\nabla u| \|_{L^p(\Omega\setminus\overline{\Omega_\eta}; W_p(\delta))}^p 
\nonumber \\ 
& =\int_{\Omega\setminus\Omega_\eta}|\nabla u(x)|^pW_p(\delta(x))\,dx.  
\end{align}
Hence, by letting $k\to\infty$ in (\ref{3.22}), 
from (\ref{ms4}), (\ref{ms5}), (\ref{3.28}) and (\ref{3.29}) it follows that 
\begin{align}\label{3.30}
J_{p,\lambda}^w & \ge \Lambda_{p}\! 
\left(1-\int_{\Omega\setminus\Omega_\eta}\frac{|u(x)|^pW_p(\delta(x))}{F_{\eta_0}(\delta(x))^p}\,dx\right)
+ s(w)L' \int_{\Sigma_\eta} |u(\sigma_\eta)|^p W_p(\eta)\,d\sigma_\eta 
\nonumber \\ 
& \quad \, + \int_{\Omega\setminus\Omega_\eta}|\nabla u(x)|^pW_p(\delta(x))\,dx 
-\lambda \int_\Omega |u(x)|^pW_p(\delta(x))\,dx. 
\end{align}
If $w(t)\in Q({\R}_+)$, then $s(w)=1$,   hence we can omit the integrand on the  surface $\Sigma_\eta$. On the other hand
if $w(t)\in P({\R}_+)$, then $\lim_{t\to+0} W_p(t)=\lim_{t\to+0}w(t)^{p-1}=0 $. 
Thus, letting $\eta\to +0$ in (\ref{3.30}), we obtain that 
\begin{align}\label{3.31}
J_{p,\lambda}^w &\ge \Lambda_{p}\! 
 \left(1-\int_{\Omega}\frac{|u(x)|^pW_p(\delta(x))}{F_{\eta_0}(\delta(x))^p}\,dx\right) 
\nonumber \\ 
& \quad\,+ \int_{\Omega}|\nabla u(x)|^pW_p(\delta(x))\,dx 
-\lambda \int_\Omega |u(x)|^pW_p(\delta(x))\,dx. 
\end{align}
Since it holds that 
\begin{equation}\label{3.32}
0<\int_\Omega |u(x)|^p\frac{W_p(\delta(x))}{F_{\eta_0}(\delta(x))^p}\,dx \le 
\liminf_{k\to\infty}\int_\Omega \frac{|u_k(x)|^pW_p(\delta(x))}{F_{\eta_0}(\delta(x))^p}\,dx =1 
\end{equation}
by $u\ne 0$, (\ref{ms1}), (\ref{ms3}) and weakly lower semi-continuity of the $L^p$-norm, 
we have from (\ref{3.31}) and (\ref{3.32}) that 
\begin{align}\label{3.33} 
J_{p,\lambda}^w
& \ge \Lambda_{p}\! \left(1-\int_{\Omega}\frac{|u(x)|^pW_p(\delta(x))}{F_{\eta_0}(\delta(x))^p}\,dx\right) 
+  \chi^w_{p,\lambda}(u) \int_{\Omega}\frac{|u(x)|^pW_p(\delta(x))}{F_{\eta_0}(\delta(x))^p}\,dx   
\nonumber \\ 
& \ge \min\bigl(\Lambda_{p},  \chi^w_{p,\lambda}(u)\bigr). 
\end{align}
This together with Lemma \ref{lem3.2} implies (\ref{3.25}). 
Moreover, by (\ref{3.25}) and (\ref{3.33}), we conclude that 
\begin{equation}\label{3.34}
J_{p,\lambda}^w
= \Lambda_{p}\! \left(1-\int_{\Omega}\frac{|u(x)|^pW_p(\delta(x))}{F_{\eta_0}(\delta(x))^p}\,dx\right) 
+  \chi^w_{p,\lambda}(u) \int_{\Omega}\frac{|u(x)|^pW_p(\delta(x))}{F_{\eta_0}(\delta(x))^p}\,dx.  
\end{equation}
In addition, if $J_{p,\lambda}^w <\Lambda_{p}$, 
then $J_{p,\lambda}^w= \chi^w_{p,\lambda}(u)$ by (\ref{3.25}), 
and so, it follows from (\ref{3.34}) and (\ref{ms1}) that 
\begin{equation}\label{3.35} 
\int_{\Omega}\frac{|u(x)|^pW_p(\delta(x))}{F_{\eta_0}(\delta(x))^p}\,dx=1
=\lim_{k\to\infty}\int_{\Omega}\dfrac{|u_k(x)|^pW_p(\delta(x))}{F_{\eta_0}(\delta(x))^p}\,dx. 
\end{equation}
(\ref{ms3}) and (\ref{3.35}) imply that  
\begin{equation}
u_k \longrightarrow u \quad \text{in} \ \ L^p(\Omega,{W_p(\delta)}/{F_{\eta_0}(\delta)^p}). 
\end{equation}
Further, by (\ref{ms1}), (\ref{ms4}), (\ref{minimizer}) and (\ref{3.35}), we obtain that 
\begin{align*}
\int_{\Omega} &|\nabla u_k(x)|^pW_p(\delta(x))\,dx
 =   \chi^w_{p,\lambda}(u_k)+\lambda \int_\Omega |u_k(x)|^pW_p(\delta(x))\,dx 
\\ 
& \longrightarrow 
\chi_\lambda^\alpha(u)+\lambda \int_\Omega |u(x)|^pW_p(\delta(x))\,dx 
=\int_{\Omega}|\nabla u(x)|^pW_p(\delta(x))\,dx. 
\end{align*}
This together with (\ref{ms2}) implies that 
\begin{equation}\label{4.55}
\nabla u_k \longrightarrow \nabla u \quad \text{in} \ \ (L^p(\Omega;W_p(\delta)))^N . 
\end{equation}
(\ref{4.55}) and (\ref{ms4}) show (\ref{3.27}). 
Consequently it completes the proof.
\end{proof}

\medskip

\noindent
{\bf  Proof of the assertion 3 of Theorem \ref{Main}.} \ 
Let $\lambda>\lambda^\ast$. 
Then $J_{p,\lambda}^w < \Lambda_{p}$ by the assertion 1 of Theorem \ref{Main}. 
Let $\{u_k\}$ be a minimizing sequence for (\ref{J2}) satisfying $(\ref{ms1}) \sim (\ref{ms5})$. 
Then we see that $u\ne 0$ by Proposition \ref{concentration}. 
Therefore, by applying Proposition \ref{compactness}, 
we conclude that $ \chi^w_{p,\lambda}(u) =J_{p,\lambda}^w$, 
namely $u$ is a minimizer for (\ref{J2}). 
It finishes the proof.
\qed

\section{Proof of  Theorem \ref{p=2}}

For  $M>0$  and  $w(t)\in W(\R_+)$, we define the following operator:
\begin{equation}
L_{M}^w(u(x)) =-\div (w(\delta(x)) \nabla u(x)) - J^w_{2,\lambda^*}\frac{w(\delta(x))u(x)}{F_{\eta_0}(\delta(x))^2}  +M w(\delta(x)) u(x).
\end{equation}
Our proof of  Theorem \ref{p=2} is relied on  the  maximum principle and the following 
non-existence result on the operator $L_{M}^w$ :
\begin{lem}\label{lem5.1} Let $\Omega$ be a bounded domain of class $C^2$ in ${\R}^N$. 
Assume  that  $w(t)\in W_A(\R_+)$ and $w(t)$ satisfies the condition  (\ref{3.3}). 
If $u(x)$ is a non-negative function in 
$W^{1,2}_0(\Omega; w(\delta))\cap C(\overline{\Omega})$ 
and satisfies the inequality 
\begin{equation}
L_{M}^w(u(x))\ge 0\qquad \mbox{ in } \ \Omega \label{positive}
\end{equation}
in the  sense of distributions for some  positive  number $M$, then $u(x)\equiv 0$.
\end{lem}

Admitting  this lemma for  the  moment, we prove Theorem \ref{p=2}.

\medskip

\noindent
{\bf Proof of Theorem \ref{p=2}.} \ 
If the infimum $J^w_{2,\lambda^*}$ in (\ref{J2}) is achieved  by  a function $u(x)$ then it is also achieved by $|u(x)|$. Therefore  there exists $u(x)\in W^{1,2}_0(\Omega;w(\delta))$, $u(x)\ge 0$ such  that
$$ -\div (w(\delta(x)) \nabla u(x)) - J^w_{2,\lambda^*}\frac{w(\delta(x))u(x)}{F_{\eta_0}(\delta(x))^2} -\lambda^* w(\delta(x)) u(x)=0.$$
By the standard  regularity theory  of the elliptic  type,  
we see that $u(x)\in C(\overline{\Omega})$, 
and by the maximum principle, $u(x)>0$ in $\Omega$. 
Then $u(x)$  clearly satisfies the inequality (\ref{positive}) for some $M>0$,
and hence the assertion of Theorem \ref{p=2} is  a consequence of Lemma \ref{lem5.1}. 
\qed

\medskip

\noindent
{\bf Proof of Lemma \ref{lem5.1}.} \ 
Assume by contradiction that there exists a non-negative function $u(x)$ as  in Lemma \ref{lem5.1}. By the  maximum principle, we  see  $u(x) >0$ in $\Omega$.
Let  us  set 
\begin{equation*}
v_s(t)= f_{\eta_0}(t)^{1/2}G_{\eta_0}(t)^{-s}\qquad \mbox{for} \ \ \ s > 1/2.
\end{equation*}
Then we have $v_s(t)\in W_0^{1,2}((0,\eta_0);w)$ 
and $v_s(\delta(x))\in W_0^{1,2}(\Omega_{\eta_0};w(\delta))$. 
We  assume that  
$\eta_0$ is sufficiently small so  that $\delta(x)\in C^2(\Omega_{\eta_0})$, 
and Theorem \ref{CT3} holds in $\Omega_{\eta_0}$. 
Since $|\nabla \delta(x)|=1$, we have for $\delta=\delta(x)$
$$ \div (w(\delta) \nabla ( v_s(\delta)))= w(\delta) v'_s(\delta)\Delta \delta
+ w'(\delta)v'_s(\delta) + w(\delta) v''_s(\delta).$$
With somewhat more calculations we  have 
\begin{align} \div (w(\delta) \nabla &(v_s(\delta))) = 
f_{\eta_0}(\delta)^{-1/2} G_{\eta_0}(\delta)^{-s}
\left( s(w)/2 + sG_{\eta_0}(\delta)^{-1}\right)\Delta \delta \notag \\
&+ w(\delta)^{-1}f_{\eta_0}(\delta)^{-3/2} G_{\eta_0}(\delta)^{-s}
\left( -1/4+ s(s+1)G_{\eta_0}(\delta)^{-2}\right).\notag
\end{align}
Since $J^w_{2,\lambda^*}=1/4$,  we  have 
\begin{align} &L_{M}^w(v_s(\delta))
 = 
-w(\delta)^{-1}f_{\eta_0}(\delta)^{-3/2} G_{\eta_0}(\delta)^{-s-2} \notag\\ 
&\times \bigl\{ s(s+1) +\Delta \delta F_{\eta_0}(\delta)\left( s(w)  G_{\eta_0}(\delta)^2/2+ sG_{\eta_0}(\delta)\right)
-M F_{\eta_0}(\delta)^{2} G_{\eta_0}(\delta)^{2} \bigr\}.\notag
\end{align}
From Lemma \ref{lem lim F=0}, 
Remark \ref{remark2.2}, 1 and (\ref{3.3}) it follows that  
\begin{equation*}
F_{\eta_0}(t), \ G_{\eta_0}(t)^{-1}, \ F_{\eta_0}(t) G_{\eta_0}(t), \ F_{\eta_0}(t) G_{\eta_0}(t)^2 \ 
\longrightarrow 0 \quad\text{as} \ \ t\to+0.
\end{equation*}
Therefore  we have 
$$ L_M^w(v_s(\delta(x)))\le 0 \qquad \mbox { in } \ \Omega_{\eta_0}.$$
Now  we choose a small  $\e>0$  so that $\e v_s(\delta(x)) \le u(x)$ on $\Sigma_{\eta_0}$, and set
$w_s(\delta(x))= \e v_s(\delta(x))-u(x)$. 
Then  $w_s^+(\delta(x)) =\max\bigl( w_s(\delta(x)),0\bigr) \in W_0^{1,2}(\Omega_{\eta_0};w(\delta))$, and we  see that
\begin{equation*}
 L_M^w(w_s(\delta(x)))\le 0 \qquad \mbox{ in } \ \Omega _{\eta_0}.
\end{equation*}
Hence  we have for $\delta=\delta(x)$
\begin{equation*}
\int_{\Omega _{\eta_0}} \left( |\nabla w_s^+(\delta)|^2 w(\delta) - \frac{ w(\delta) w_s^+(\delta)^2}{4 F_{\eta_0}(\delta)^2}
+M  w(\delta) w_s^+(\delta)^2\right)\,dx \le 0.
\end{equation*}
But, by Theorem \ref{CT3}, we  have
\begin{equation*}
\int_{\Omega_{\eta_0}} \left( |\nabla w_s^+(\delta(x))|^2 w(\delta(x))
  - \frac{w(\delta(x))w_s^+(\delta(x))^2}{4F_{\eta_0}(\delta(x))^2}  \right)\,dx
\ge 0.
\end{equation*}
Therefore we have $w_s^+(\delta(x))=0$ in $\Omega_{\eta_0}$, and so 
$\e v_s(\delta(x)) \le u(x)$ in $\Omega_{\eta_0}$ for any $s>1/2$. 
By letting $s\to 1/2$,
$ \e f_{\eta_0}(\delta(x))^{1/2} G_{\eta_0}(\delta(x))^{-1/2} 
\le u(x)$ holds in $\Omega_{\eta_0}$. Namely
\begin{equation*}
\frac{ u(x)^2 w(\delta(x))}{F_{\eta_0}(\delta(x))^2}\ge \e^2 \frac{1}{ F_{\eta_0}(\delta(x)) G_{\eta_0}(\delta(x))} 
\quad \text{in} \ \ \Omega_{\eta_0}.
\end{equation*}
Since it  holds  that 
$\bigl(F_{\eta_0}(\delta(x)) G_{\eta_0}(\delta(x))\bigr)^{-1}\notin 
L^1(\Omega_{\eta_0})$ by Remark \ref{remark2.2}, 1, 
we have that 
$u(x)\notin L^2(\Omega_{\eta_0};w(\delta)/F_{\eta_0}(\delta)^2)$. 
This together with Hardy's inequality (\ref{2.11})  
contradicts to that $u(x)\in W^{1,2}_0(\Omega;w(\delta))$. 
\qed

\bigskip\bigskip

\noindent
Hiroshi Ando:\\
Department of Mathematics, Faculty of Science, Ibaraki University,\\
Mito, Ibaraki, 310-8512, Japan;\\
hiroshi.ando.math@vc.ibaraki.ac.jp

\medskip

\noindent
Toshio Horiuchi:\\
Department of Mathematics, Faculty of Science, Ibaraki University,\\
Mito, Ibaraki, 310-8512, Japan;\\
toshio.horiuchi.math@vc.ibaraki.ac.jp

\end{document}